\documentclass{amsart}[11pt]
\usepackage[all]{xy}
\usepackage{amsmath}
\usepackage{amsfonts}
\usepackage{amsthm}
\usepackage{xspace}
\usepackage{amssymb}
\usepackage{enumerate}
\usepackage{graphics}
\usepackage{graphicx}
\usepackage[colorlinks=true,hyperindex]{hyperref} 

\theoremstyle{plain}
\newtheorem{thm}{Theorem}
\newtheorem{lem}[thm]{Lemma}
\newtheorem{cor}[thm]{Corollary}
\newtheorem{prop}[thm]{Proposition}
\newtheorem{conj}[thm]{Conjecture}

\theoremstyle{definition}

\newtheorem{rem}[thm]{Remark}
\newtheorem{const}[thm]{Construction}
\numberwithin{thm}{section}

\newcommand{\Hom}{\operatorname{Hom}}

\newcommand{\Thick}{{\sf Thick}}
\newcommand{\Ext}{\operatorname{Ext}}

\newcommand{\pd}{\operatorname{pd}}

\newcommand{\B}{{\sf B}}
\newcommand{\C}{{\sf C}}
\newcommand{\D}{{\sf D}}
\newcommand{\DR}{\D^b(R)}
\newcommand{\DS}{\D^b(S)}
\newcommand{\T}{{\sf T}}

\newcommand{\Sig}{{\sf \Sigma}}
\newcommand{\Rmod}{{R \textrm{-mod}}}

\newcommand{\comment}[1]{}

\begin{document}

\title[A Generalization of the Auslander-Reiten Conjecture for $\D^b(R\textrm{-}\rm{\lowercase{mod}})$]{A Generalization of the Auslander-Reiten Conjecture for the Bounded Derived Category}
\author{Kosmas Diveris \ and \ Marju Purin}\address{Department of Mathematics\\Syracuse University\\Syracuse\\ NY 13244\\USA}\email{kjdiveri@syr.edu}
\address{Department of Mathematics\\Manhattan College\\ Riverdale\\ NY 10471\\ USA}\email{marju.purin@manhattan.edu}

\date{August 27, 2011}
\maketitle 

\begin{abstract} 
We study the bounded derived category $\D^b(\Rmod)$ of a left Noetherian ring $R$. We give a version of the Generalized Auslander-Reiten Conjecture for $\D^b(\Rmod)$ that is equivalent to the classical statement for the module category and is preserved under derived equivalence. 
\end{abstract}

\begin{section}{Introduction}
In this paper we enter the realm of some homological conjectures that have motivated much research in ring theory. We begin by recalling these conjectures and giving a brief overview of their status. Throughout this note $R$ is a left Noetherian ring and $\Rmod$ denotes the category of finitely generated left $R$-modules. 

Recall the statement of the Auslander-Reiten Conjecture \cite{AR}:

\begin{conj} [Auslander-Reiten Conjecture] \label{conj:arc} If $M$ is an $R$-module such that $\Ext_R^i(M,M) = 0 = \Ext_R^i(M,R)$ for all $i \geq 1$, then $M$ is projective.  \end{conj}


Counterexamples to this conjecture were provided by R. Schulz in the form of algebras over a skew field $k$:  \begin{align} \label{eq:schulz} k \left< X,Y \right> / (X^2, Y^2, XY- cYX) \end{align} where $c$ is a suitably chosen element of $k$ (see Section 3 of \cite{S}). The question remains open, however, for finite-dimensional algebras over a field, artin algebras, and commutative rings.
In fact, M. Auslander and I. Reiten showed that for artin algebras this conjecture is equivalent to the Generalized Nakayama Conjecture \cite{AR}. 

A natural generalization of the conjecture is the following:
\begin{conj} [Generalized Auslander-Reiten Conjecture] \label{conj:garc} If $M$ is an $R$-module such that $\Ext_R^i(M,M) = 0 = \Ext_R^i(M,R)$ for all $i \gg 0$, then $M$ has finite projective dimension.  \end{conj}

In \cite{S}, R. Schulz also provided a counterexample to this conjecture in the form of a finite-dimensional self-injective algebra over a field. In fact, this algebra is defined by the same relations as in (\ref{eq:schulz}) where $k$ is taken to be a field and $c$ is any element in $k$ that is not a root of unity.  The conjecture holds for several classes of rings, for example, for group algebras of finite groups (Lemma 5.2.3 in \cite{B}), and complete intersections (Proposition 4.8 in \cite{AB}). The conjecture still remains open for some classes of rings, for example, for commutative rings.  Note that if a ring satisfies Conjecture\autoref{conj:garc}, then it also satisfies Conjecture\autoref{conj:arc}. 

While the Generalized Auslander-Reiten Conjecture fails for artin algebras in general, it is known that if an artin algebra satisfies Conjecture\autoref{conj:garc}, then it satisfies the Gorenstein Symmetry Question \cite{AR2,LH,D} (the latter is a meaningful statement only for algebras of finite injective dimension).  

 J. Wei has shown that for artin algebras both the Generalized and original Auslander-Reiten Conjectures are stable under a tilting equivalence \cite{W, W2}. Since a tilting equivalence is a special case of a derived equivalence, the author then asks whether these conjectures are stable under any derived equivalence. We consider all Noetherian rings and answer J. Wei's question in the affirmative for the Generalized Auslander-Reiten Conjecture. That is, any two derived equivalent Noetherian rings satisfy Conjecture\autoref{conj:garc} simultaneously. 

The argument in \cite{W} uses torsion theory which is well-understood in the setting of artin algebras. Our approach is to give a version of the Generalized Auslander-Reiten Conjecture for the bounded derived category of any Noetherian ring which we then show to be equivalent to the classical statement for the module category in Theorem\autoref{thm:equiv}. Next we show that the derived version of the conjecture passes through an equivalence of derived categories in Theorem\autoref{thm:derived}. Finally, we conclude in Corollary\autoref{cor:conj_is_stable} that any two derived equivalent Noetherian rings satisfy the classical Generalized Auslander-Reiten Conjecture simultaneously.


\section{Preliminaries}
In this section we recall definitions, set up notation, and collect some lemmata.

Let $\T$ be a triangulated category with shift $\Sig$. A non-empty full subcategory ${\sf S}$ of $\T$ is {\it triangulated} if it satisfies the conditions: \begin{enumerate}
\item If $X$ is in $\sf S$, then $\Sig^i X$ is in $\sf S$ for all $i \in \mathbb{Z}$.
\item If two of the terms in $\left\{ X, Y, Z \right\}$ from the triangle
\begin{align*} X \rightarrow Y \rightarrow Z \rightarrow \Sig X \end{align*} in $\T$ belong to ${\sf S}$, then remaining term also belongs to ${\sf S}$.  
\end{enumerate}

A triangulated subcategory ${\sf S}$ is {\it thick} if, in addition, it is closed under direct summands. If $\C$ is a collection of objects in a triangulated category $\T$, its {\it thick closure}, $\Thick(\C)$, is the intersection of all thick subcategories of $\T$ containing all of the objects in $\C$. 

Two objects $M$ and $N$ in $\T$ are {\it eventually orthogonal}, and we write $M \perp N$, if $\Hom_{\T}(M,\Sig^iN) = 0$ for all $|i| \gg 0$. Note that we do not require vanishing for every $i \in \mathbb{Z}$, as some authors do, when using this notation.  For a subcategory $\sf{C}$ of a triangulated category $\T$ we set  \begin{align*} ^{\perp}{\sf C} = \{ M \in \T \ | M \perp C \ {\rm for \ all} \ C \in \sf{C} \} , \ {\rm and} \end{align*} \begin{align*}  {\sf C}^{\perp} = \{ N \in \T \ | C \perp N \ {\rm for \ all} \ C \in \sf{C} \} \end{align*} If $\B$ is another subcategory we write $\B \perp \C$ if $\B \subseteq {^{\perp} \C}$, or equivalently, $\C \subseteq \B^{\perp}$. In this case, we say that the subcategories $\B$ and $\C$ are eventually orthogonal.

Our first observation is that the subcategories $^{\perp} \C$ and $\C^{\perp}$ are thick. 

\begin{lem} \label{lem:PerpThick} For any subcategory ${\sf C}$ of $\T$,  $^{\perp}{\sf C}$ and ${\sf C}^{\perp}$ are thick subcategories of $\T$. \end{lem}

\begin{proof}  We show that $ {\sf C}^{\perp}$ is a thick subcategory, an analogous argument gives the other claim.  It is immediate from the definition that ${\sf C}^{\perp}$ is closed under shifts.  Suppose now that \begin{align} \label{SES} X \rightarrow Y \rightarrow Z \rightarrow \Sig X \end{align} is a triangle in $\T$.  The functor $\Hom_{\T}(C,-)$ applied to (\ref{SES}) gives a long exact sequence of abelian groups: \begin{align} \label{les}  \cdots \to  \Hom_{\T}(C,\Sig ^i X) &\rightarrow \Hom_{\T}(C, \Sig^i Y) \\  &\rightarrow  \Hom_{\T}(C,\Sig^i Z)  \rightarrow 
 \Hom_{\T}(C,\Sig ^{i+1} X)  \to \cdots \notag \end{align}  
Therefore, if any two of the terms $\{X,Y,Z\}$ in $\C^{\perp}$, then so is the third.  This shows that ${\sf C}^{\perp}$ is a triangulated subcategory of $\T$.

Since the equality $\Hom_{\T}(C, X \oplus Y) = \Hom_{\T}(C, X) \oplus \Hom_{\T}(C,Y)$ holds for all $C,X,Y \in \T$, vanishing on the left side gives vanishing of both of the terms on the right. From this it follows that ${\sf C}^{\perp}$ is closed under direct summands and therefore it is a thick subcategory of $\T$.
\end{proof}

The following observation is used repeatedly in the sequel. It shows that eventual orthogonality is preserved by taking thick closures. 

\begin{lem} \label{lem:ThickPerp} If ${\sf B}$ and ${\sf C}$ are subcategories of $\T$ such that  ${\sf B} \perp {\sf C}$, then \begin{align*} \Thick({\sf B}) \perp \Thick({\sf C}). \end{align*} 
\end{lem}

\begin{proof}   Since $\B \perp \C$ holds, we have a containment
 $\B \subseteq {^{\perp}}\C$.  By Lemma \autoref{lem:PerpThick}, $^{\perp} \C$ is a thick subcategory of $\T$. But $\Thick(\B)$ is a subset of all thick subcategories of $\T$ that contain $\B$ whence  $\Thick(\B) \subseteq {^{\perp}\C}$, or equivalently, $\C \subseteq \Thick(\B)^{\perp}$. Employing Lemma\autoref{lem:PerpThick} again, gives that $\Thick(\B)^{\perp}$ is also thick. The same argument now yields $\Thick(\C) \subseteq \Thick(\B)^{\perp}$, i.e. $\Thick(\B) \perp \Thick(\C)$. \end{proof}

In what follows we restrict our attention to a particular triangulated category. Namely, the bounded derived category $\DR = \D^b(\Rmod)$ of finitely generated left $R$-modules where $R$ is a left Noetherian ring. We make a couple of observations about a certain thick subcategory of $\DR$ that become useful to us later.  The objects in the subcategory $\Thick(R)$ of $\DR$ are precisely the {\it perfect} complexes, that is the complexes isomorphic to finite complexes of finitely generated projective modules (See Section 1.2 in \cite{Bu}).  Second, an equivalence of derived categories restricts to an equivalence of the corresponding subcategories of perfect complexes. The latter can be extracted from the work J. Rickard (See Theorem 6.4 and Propositions 8.2 and 8.3 in \cite{R}). We include a proof for completeness:

\begin{lem} \label{lem:perfects} Let $F : \DR \to \DS$ be an equivalence of triangulated categories. Then $F$ restricts to an equivalence between the subcategories $\Thick(R)$ and $\Thick(S)$.
\end{lem}

\begin{proof} We first show that $\Thick(R) = {^{\perp}\D^b(R)}$.  Since $\Hom_{\DR}(R,\Sig^iM) = H_{-i}(M)$, we see that  $R \perp M$ for all $M \in \D^b(R)$.  Lemma\autoref{lem:ThickPerp} now gives the containment $\Thick(R) \subseteq {^{\perp}\DR}$.  To see the opposite containment,  consider $\Rmod \subseteq \DR$ as the subcategory consisting of the stalk complexes of degree 0.  Then we see that ${^{\perp}\DR} \subseteq {^{\perp}\Rmod}$.  Now recall that $Q$ is a perfect complex if and only if $\Hom_{\DR}(Q,\Sig^iM) = 0$ for all $R$-modules $M$ and $i \gg 0$ (see Lemma 1.2.1 of \cite{Bu}, for example). This gives that ${^{\perp}\DR} \subseteq  \Thick(R)$.

Similarly, we have $\Thick(S) = {^{\perp}\DS}$.  Now observe that $M \perp \DR$ if and only if $F(M) \perp \DS$, so that $F(\Thick(R)) = \Thick(S)$ as desired.
\end{proof}

We close by recording a connection between extensions over $R$ and morphisms in $\DR$.

\begin{rem} \label{rem:equiv} For any $X,Y \in \Rmod$ and every integer $i$, we have an isomorphism of abelian groups $\Ext_{R}^i(X,Y) \cong \Hom_{\DR}(X,\Sig^iY)$ (See Section 1.5 in \cite{Kr}, for example).  Here we use the convention that $\Ext^i_R(X,Y) = 0$ when $i<0$.  From this we see that the following conditions are equivalent for an $R$-module $M$: 
\begin{enumerate}[(i)]
\item \label{mod} $\Ext_R^i(M,M) = 0 = \Ext_R^i(M,R) $ for all $|i| \gg 0$.
\item \label{cx} In $\DR$, we have $M \perp M$ and  $M \perp R$.
\end{enumerate}
\end{rem}


\section{The Generalized AR Conjecture for $\D^b(R\textrm{-}\rm{\lowercase{mod}})$}
This section contains our main results. In it we give the derived category version of the Generalized Auslander-Reiten Conjecture, show that it is equivalent to the classical conjecture for the module category, and prove that it is stable under a derived  equivalence.

Recall that $\DR $ denotes the bounded derived category of finitely generated left $R$-modules where $R$ is a left Noetherian ring.  We now state the promised derived category version of Conjecture\autoref{conj:garc}.

\begin{conj}  \label{conj:GARC} If $M \in \DR$ is such that $M \perp M$ and $M \perp R$, then $M$ belongs to $\Thick(R)$. \end{conj}

In order to show that Conjecture \ref{conj:GARC} follows from Conjecture \ref{conj:garc} we need to pass from the derived category to the module category. The following construction describes how we associate a module $\Omega M$ to a complex $M$. In the case when $M$ is a module (viewed as a stalk complex in $\DR$), $\Omega M$ is its first syzygy module.  This observation motivates the construction. 

The {\it infimum} of a complex $M$, denoted $\inf M$, is the lowest degree where homology appears. Similarly, the {\it supremum} of a complex $M$, denoted $\sup M$, is the highest degree where homology appears. 

\begin{const}\label{syzygy}

Fix a complex $M \in \DR$ with $\inf M = 0$ and set $n = \sup M$.  Let $P$ be a projective resolution of $M$, i.e., $P$ is a complex of projective modules quasi-isomorphic to $M$, and consider the following truncations of $P$:
\begin{align*}  P_{\leq n} =  0 \rightarrow P_{n} \rightarrow \cdots \rightarrow P_0 \rightarrow 0  \end{align*}  
\begin{align*} P_{\geq n+1} = \cdots \rightarrow P_i \rightarrow \cdots \rightarrow P_{n+2} \rightarrow P_{n+1}  \rightarrow 0   \end{align*}
which fit into a short exact sequence of complexes 
\begin{align} \label{ses1} 0 \rightarrow P_{\leq n} \rightarrow P \rightarrow P_{\geq n+1} \rightarrow 0\end{align}  

We observe that $P_{\leq n}$ is a finite complex of projective modules, so that it is in  $\Thick(R)$.  Since $H_i(P) = H_i(M) =0$ for $i>n$, we see that $P_{\geq n +1} \cong H(P_{\geq n +1}) \cong \Sig^{n+1} {\rm coker}(\partial_{n+2})$ in $\DR$.  We set $\Omega M =  {\rm coker} (\partial_{n+2}) \in \Rmod$.   The short exact sequence (\ref{ses1}) gives rise to a triangle: 

\begin{align} \label{ses} Q \rightarrow M \rightarrow \Sig^{n+1} \Omega M \rightarrow \Sig Q \end{align} in $\DR$ with $Q \in \Thick(R)$.  

For a general complex $M \in \DR$, with $m = \inf M$, we set $\Omega M = \Sig^m \Omega (\Sig^{-m}M)$ and obtain a triangle as in (\ref{ses}) above (in this case $ n = \sup M - \inf M$).

\end{const}

The key step in showing that Conjectures \ref{conj:garc} and \ref{conj:GARC} are equivalent is the following observation.

\begin{prop} \label{Cx-Mod} If $M \perp M$ and $M \perp R$, then $\Omega M \perp \Omega M$ and $\Omega M \perp R$.  \end{prop}

\begin{proof}  In order to show $\Omega M \perp \Omega M$, we first use the triangle (\ref{ses}) from Construction\autoref{syzygy} to obtain $M \perp \Omega M$. We then use this to show that $\Omega M \perp \Omega M$.

    Since $R$ belongs to  $M^{\perp}$, Lemma\autoref{lem:ThickPerp} gives that $\Thick(R)$ is contained in $M^{\perp}$. In particular, we see that $M ^{\perp}$ contains $Q$.  Now, in Lemma \autoref{lem:PerpThick} we showed that $M^{\perp}$ is a triangulated subcategory of $\DR$. Therefore, since the first two terms of the triangle in (\ref{ses}) are in $M^{\perp}$, so is its third term, $\Sig^{n+1} \Omega M$.  As $M^{\perp}$ is closed under shifts, we obtain that $\Omega M$ is in  $M^{\perp}$, or equivalently, that $M$ is in ${^{\perp} \Omega M}$.

 Now recall that $\Thick(R)^{\perp} = \DR$ whence $Q \perp \Omega M$.  This gives that $Q$ belongs to ${^{\perp}\Omega M}$, and thus the first two terms of the triangle (\ref{ses}) are in $^{\perp} \Omega M$. It follows that all terms of the triangle (\ref{ses})  must belong to $^{\perp} \Omega M$. Since $^{\perp} \Omega M$ is also closed under shifts, we have $\Omega M \perp \Omega M$.

Finally, since both $M$ and $Q$ are in ${^{\perp}R}$, the triangle (\ref{ses}) also gives that $\Omega M$ is also in  ${^{\perp}R}$, i.e. $\Omega M \perp R$.  
\end{proof}

The next theorem shows that the derived category version of the Generalized Auslander-Reiten Conjecture is equivalent to the classical statement for the module category.  What is somewhat surprising is that if the Generalized Auslander-Reiten Conjecture holds for the module category, then it extends to the entire derived category.

\begin{thm} \label{thm:equiv} Conjecture \ref{conj:garc} holds for $R\textrm{-}\rm{mod}$ if and only if Conjecture \ref{conj:GARC} holds for $\D^b(R \mbox{-}{\rm mod})$. \end{thm}

\begin{proof} 
Assume that Conjecture \autoref{conj:GARC} holds and $M \in \Rmod$ satisfies the hypotheses $\Ext_R^i(M,M) = 0 = \Ext_R^i(M,R) $ for all $|i| \gg 0$. Then (\ref{cx}) of Remark\autoref{rem:equiv} holds and $M$ belongs to $\Thick(R)$ by Conjecture \ref{conj:GARC}.  But this means that $\pd(M)$ is finite. 

For the converse, assume that Conjecture \ref{conj:garc} holds and $M$ in $\DR$ satisfies the hypotheses: $M \perp M$ and  $M \perp R$  in  the bounded derived category $\DR$. Proposition \ref{Cx-Mod} then gives that $\Omega M \perp \Omega M$ and $\Omega M \perp R$ also hold in $\DR$.  Next, Remark\autoref{rem:equiv} together with Conjecture \ref{conj:garc} gives that $\pd(\Omega M)$ is finite so that $\Omega M \in \Thick(R)$.  Since $\Thick(R)$ is a triangulated subcategory of $\DR$, and the first and third terms of the triangle (\ref{ses}) belong to $\Thick(R)$, it follows that its middle term, $M$, also belongs to $\Thick(R)$.
\end{proof}

\comment{

Wei has shown, in \cite{Wei}, that the Generalized Auslander-Reiten is stable under tilting for Artin algebras.  As tilting is a special case of a derived equivalence, Wei asks if the Generalized Auslander-Reiten conjecture is stable under any derived equivalence?  We show that this question has an affirmative answer.  

Wei's proof uses torsion theory that arises from tilting.  Our proof uses the derived version of the Generalized Auslander-Reiten conjecture and the following result of Rickard, and works for all Noetherian rings.  The equivalence of conditions (1), (2) and (3) in Theorem \ref{rickard} are contained in Theorem 6.4 of \cite{Rickard}.  That condition (4) is equivalent to (1) follows from Propositions 8.2 and 8.3 of \cite{Rickard}, since we are only considering Noetherian rings.  We denote by ${\rm Proj}\mbox{-}R$ the category of (not necessarily finitely generated) projective $R$-modules, so that $\Thick({\rm Proj}\mbox{-}R)$ consists of all complexes isomorphic to bounded complexes of arbitrary projective $R$-modules.

The following theorem is extracted from Theorem 6.4 and Propositions 8.2 and 8.3 of \cite{R}.

\begin{thm}[Rickard]  \label{thm:rickard} The following conditions are equivalent for any two Noetherian rings $R$ and $S$:
\begin{enumerate}[\rm(i)] 
\item  $\D^b(R \mbox{-}{\rm mod}) \simeq \D^b(S \mbox{-}{\rm mod})$
\item $\Thick(R) \simeq \Thick(S)$
\end{enumerate}
Moreover, the equivalence in {\rm(ii)} is given by restricting the equivalence in {\rm(i)}.
\end{thm} 
} 

We say that two rings $R$ and $S$ are derived equivalent if there exists an equivalence of triangulated categories between $\D^b(\Rmod)$ and $\D^b(S \mbox{-}{\rm mod})$.  In the next theorem, we show that Conjecture \ref{conj:GARC} is stable under a derived equivalence of Noetherian rings.

\begin{thm} \label{thm:derived} If $R$ and $S$ are derived equivalent and Conjecture \ref{conj:GARC} holds for $R$, then it also holds for $S$.  \end{thm}

\begin{proof}  Suppose that Conjecture \ref{conj:GARC} holds for $R$ and let $F: \D^b(S\mbox{-}{\rm mod}) \rightarrow \D^b(R\mbox{-}{\rm mod})$ be an equivalence.   Let $N \in \DS$ satisfy $N \perp N$ and $N \perp S$.  By Lemma\autoref{lem:ThickPerp}, we then have that $N$ belongs to ${^{\perp}}\Thick(S)$. 

Since $R \in \Thick(R)$, we have $R = F(P)$ for some $P \in \Thick(S)$, by Lemma \autoref{lem:perfects}.  Then, for all $|i| \gg 0$ we have  \begin{align}  \Hom_{\DR}(F(N), \Sig^i(F(N)\oplus R)) =  \Hom_{\DS}(N, \Sig^i(N \oplus P) ) = 0. \end{align} 
That is, $F(N) \perp F(N)$ and $F(N) \perp R$.

Since Conjecture \ref{conj:GARC} holds for $R$ and $F(N) \perp F(N)$ and $F(N) \perp R$, it follows that $F(N)$ is in $\Thick(R)$.  Another application of Lemma\autoref{lem:perfects} gives that $N$ is in $\Thick(S)$, so that Conjecture \ref{conj:GARC} holds for $S$.
\end{proof}

We can now answer Wei's question regarding the stability of the Generalized AR Conjecture under any derived equivalence. Furthermore, we extend his result to all Noetherian rings.

\begin{cor} \label{cor:conj_is_stable}The Generalized Auslander-Reiten Conjecture is stable under a derived equivalence of Noetherian rings.  \end{cor}
\begin{proof} Theorem\autoref{thm:equiv} shows that the Conjectures \ref{conj:garc} are \ref{conj:GARC} equivalent. The corollary now follows from Theorem\autoref{thm:derived}.
\end{proof}

\end{section}

\end{document}